\documentclass[11pt,reqno]{amsart}

\usepackage{booktabs}                                 
\usepackage[english]{babel}
\usepackage{latexsym}
 \usepackage[utf8]{inputenc}
\usepackage{babel}
\usepackage{fancyhdr}
\usepackage{longtable}
\usepackage{mathrsfs}
\usepackage{geometry}
\usepackage{comment}
\usepackage{textcomp}
\usepackage{caption}
\usepackage{lmodern}
\usepackage{mathrsfs}
 \usepackage[T1]{fontenc}
 \usepackage{longtable}

\usepackage[all,cmtip]{xy}
\usepackage{epsfig}
 \usepackage{amsmath,amsfonts,amssymb,amsthm}
\usepackage{graphics}
\usepackage{graphicx}
\usepackage{verbatim}
\usepackage{dsfont}
\usepackage{enumerate}  
\usepackage{pdflscape}
\usepackage{longtable}
\usepackage{booktabs}
\usepackage{colortbl}
\usepackage[shortlabels]{enumitem}
\usepackage[breaklinks]{hyperref} 
\hypersetup{
	colorlinks,
	linkcolor={red},
	citecolor={blue},
	urlcolor={red}
}
\usepackage{stmaryrd}
\usepackage{multirow}

\usepackage[disable]{todonotes}

\newcommand{\evnrow}{\rowcolor[gray]{0.95}}

\newcommand{\QQ}{\mathbb{Q}}

\newcommand{\mQ}{\mathcal{Q}}

\newcommand{\mF}{\mathcal{F}}

\newcommand{\PP}{\mathbb{P}}

\newcommand{\of}{\mathcal{O}}

\DeclareMathOperator{\Gr}{Gr}

\DeclareMathOperator{\Bl}{Bl}

\newtheorem{thm}{Theorem}[section]

\newtheorem{proposition}[thm]{Proposition}

\newtheorem*{aim*}{Aim of this paper}

\theoremstyle{definition}

\makeatletter    
\def\l@subsection{\@tocline{1}{0,2pt}{2pc}{8mm}{\ \ }} 
\makeatletter    
\def\l@section{\@tocline{1}{0,2pt}{2pc}{8mm}{\ \ }}

\author{Enrico Fatighenti}
\address{Dipartimento di Matematica\\
Universit\`a di Bologna\\
 Piazza di Porta San Donato 5, 40126 Bologna, Italy.}
\email[E.~Fatighenti]{enrico.fatighenti@unibo.it}

\title[K3 structures from singular Fano varieties]{K3 structures from singular Fano varieties}

\begin{document}
\vspace{-0.3cm}

\begin{abstract}
We survey some results obtained in our quest for Fano varieties of K3 type and discuss why exploring the singular world might be interesting for discovering new K3 structures.
\end{abstract}

\maketitle

\vspace{-1cm}
\section{Introduction}
Fano varieties of K3 type (FK3 for short), are a very special class of smooth projective varieties that lies at the crossroad of many different topics in algebraic geometry, e.g. \emph{rationality questions} and \emph{hyperk\"ahler geometry}. In short, their peculiar name comes from the fact that they contain a Hodge-theoretical ``core'' that looks like the one of a K3 surface. We refer to \cite{survey} (and references therein) for a survey of their properties and recent advances in their classification. In \cite{bfmt}, we produced a list of 64 new FK3 fourfolds in homogeneous variety, as zero loci of suitable vector bundles. We believe that a close analysis of these variety could lead to many interesting constructions\footnote{It is worth noting that \textbf{K3-35} appeared very recently in connection with a hyperk\"ahler 6-fold associated to an EPW cube, see \cite{kkm} and also \cite{bfk+24}.}. Our analysis in \emph{op.cit.} lead us to the study of certain singular Fano fourfolds, whose resolution is a FK3. In this short survey we recap some result on the topic and give a glimpse of what our future directions may be.

\vspace{-0.5cm}

\section{K3 structures from singular varieties}

\subsection{FK3 fourfolds with singular projections}
Many of the FK3 fourfolds constructed in \cite{bfmt} - of which we follow the notations - admit a fairly simple birational geometry: they are in fact blow up of either a cubic or a Gushel-Mukai fourfold, or to other fairly simple Fano varieties with center a K3 surface. There are four exceptions, which are fairly mysterious and with a potentially rich geometry, both from the rationality and the hyperk\"ahler point of view. In our paper we named them as \textbf{R-61, R-62, R-63, R-64} - we invite the interested reader to look in our paper the exact vector bundles (of rank 6,9,8,6) that defines them.
All these varieties admit a map towards a smaller dimensional projective space which is either a DP$_5$ fibration or a conic bundle, together with another natural projection, which is birational towards a singular fourfold. In some sense, these singular fourfolds must therefore be a \emph{primitive} equivalent (albeit not \emph{stricto sensu}) of the cubic and GM-fourfolds, and we were lead to their analysis.
Other FK3 fourfolds where one projection is towards a singular fourfolds are \textbf{C-5, C-6, C-8, K3-27, K3-55}.
The most common singular fourfold involved is a complete intersection of three quadrics, in some special configurations. We will focus on this case for the rest of this survey.

\vspace{-0.1cm}
\subsection{The (singular) complete intersection of three quadrics}
Let us focus on the case of a singular complete intersection of three quadrics $\QQ= Z(Q_1, Q_2, Q_3) \subset \PP^{2n+1}$, with $n=2,3$. The singular locus of each quadric $Q_i$ coincides with $\PP(\ker(Q_i))$. We will do an extra simplifying assumption, i.e. that $\ker(Q_i)$, whenever non-empty, coincide. This assumption is in fact restrictive, but will allow us to easily realize $Y=\Bl_{\mathrm{sing} (\QQ)} \QQ$ as $(G, \mF)$, and therefore to compute their Hodge structure using the tools of \cite{dft, bfmt}. For example, \textbf{C-8} will \emph{not} satisfy this assumption.

In fact, it can be described as the blow up of $\QQ$, singular on a line, hence as a subvariety of $\PP^5 \times \PP^7$, cut by $\mQ_{\PP^5}(0,1) \oplus \of(1,1) \oplus \mF$, where $\mF$ is globally generated, has rank 2 and $c_1(\mF)=(3,1)$. It is not split (i.e., $\of(2,0) \oplus \of(1,1)$), since the resulting Fano in this case would not be of K3 type (in fact, having an $H^4$ with dimension vector $(2,32,2)$).

It is not important here that $Y$ itself is a Fano variety: the main purpose of this short note is to highlight how a Hodge structure of a smooth Fano variety which is \emph{not} a FK3 can degenerate enough to become a K3-structure when the resulting singular variety $\QQ$ is resolved. The variety $Y$ is therefore just an auxiliary object, that allows us to detect where these \emph{FK3-degenerations} might be.

In the table below, we will write down which are the $Y$ in the above hypotheses, in dimension 4 and 6, encoding the degeneration data as well. $Y$ will be described in the ambient variety $G$ as $Y=(G, \mF)$, with $\mF$ being the direct sum of $\mQ(0,1)$ together with the three quadrics of specified type. Notice how in general a quadric of type $(1,1)$ will contain the singular locus, while a quadric of type $(2,0)$ will cut it. We also avoid the case where all $Q_i$ are singular, to avoid cases where $\QQ$ is simply a cone over a lower-dimensional intersection of quadrics.

\begin{longtable}{cccccc}
\caption{4 dimensional $Y$ of K3 type}\label{tab:Y4}\\
\toprule
\multicolumn{1}{c}{G}&\multicolumn{1}{c}{$\dim(\ker(Q_i))$}& \multicolumn{1}{c}{$Q_i$-type} & \multicolumn{1}{c}{sing$(\QQ)$} &\multicolumn{1}{c}{$H^4 \neq 0$}& \multicolumn{1}{c}{Fano} \\
\cmidrule(lr){1-1}\cmidrule(lr){2-2}\cmidrule(lr){3-3} \cmidrule(lr){4-4} \cmidrule(lr){5-5} \cmidrule(lr){6-6} \endfirsthead
\multicolumn{5}{l}{\vspace{-0.25em}\scriptsize\emph{\tablename\ \thetable{} continued from previous page}}\\
\midrule
\endhead
\multicolumn{5}{r}{\scriptsize\emph{Continued on next page}}\\
\endfoot
\bottomrule
\endlastfoot
\evnrow $\PP^6 \times \PP^7$&$(1,1,0)$&$(2,0), (2,0), (1,1)$&$p$&$(1,28,1)$& K3-27\\
$\PP^5 \times \PP^7$&$(2,2,0)$&$(2,0), (2,0), (0,2)$&$2p$&$(1,24,1)$& R-63\\
\evnrow $\PP^4 \times \PP^7$&$(3,3,0)$&$(2,0), (2,0), (0,2)$&$v_2(\PP^1)$&$(1,22,1)$&  $\times $\\
$\PP^3 \times \PP^7$&$(4,0,0)$&$(2,0), (0,2), (1,1)$&$\PP^1 \times \PP^1$&$(1,22,1)$& $\times$\\

\end{longtable}
\vspace{-0.2cm}
\begin{longtable}{cccccc}
\caption{6-dimensional $Y$ of K3 type}\label{tab:Y4}\\
\toprule
\multicolumn{1}{c}{G}&\multicolumn{1}{c}{$\dim(\ker(Q_i))$}& \multicolumn{1}{c}{$Q_i$-type} & \multicolumn{1}{c}{sing$(\QQ)$} &\multicolumn{1}{c}{$H^4 \neq 0$}& \multicolumn{1}{c}{Fano} \\
\cmidrule(lr){1-1}\cmidrule(lr){2-2}\cmidrule(lr){3-3} \cmidrule(lr){4-4} \cmidrule(lr){5-5} \cmidrule(lr){6-6} \endfirsthead
\multicolumn{5}{l}{\vspace{-0.25em}\scriptsize\emph{\tablename\ \thetable{} continued from previous page}}\\
\midrule
\endhead
\multicolumn{5}{r}{\scriptsize\emph{Continued on next page}}\\
\endfoot
\bottomrule
\endlastfoot
\evnrow $\PP^7 \times \PP^9$&$(2,2,0)$&$(2,0), (2,0), (1,1)$&$\PP^1$&$(1,22,1)$& \checkmark $(\iota_Y=2)$\\
$\PP^5 \times \PP^9$&$(4,4,0)$&$(2,0), (2,0), (0,2)$&$\PP^1 \times \PP^1$&$(1,24,1)$&\checkmark $(\iota_Y=1)$ \\
\evnrow $\PP^4 \times \PP^9$&$(5,5,0)$&$(2,0), (2,0), (0,2)$&$Q^3$&$(1,22,1)$&  $\times$\\
$\PP^3 \times \PP^9$&$(6,0,0)$&$(2,0), (0,2), (1,1)$&$\Gr(2,4)$&$(1,22,1)$& $\times$\\

\end{longtable}

Birational description can be given in the 6-fold case as well. Consider the first case as an example: we can equivalently describe $Y$ as $Y=(\PP_Z(\of(-1)^2 \oplus \of(-2), \of_P(1))$, where $Z \subset \PP^7$ is given by the smooth complete intersection of two quadrics. The projection $Y \to Z$ is generically a $\PP^1$-bundle, with special fibers $\PP^2$ over an octic K3 surface $S$. Hence, by the so-called \emph{Cayley trick} we can describe both the Hodge structure and the derived category of $Y$ in terms of the one of $S$, see \cite[Prop. 46]{bfm}.

\subsubsection{Singular $(2,2)$ divisors in $\PP^2 \times \PP^3$ and rationality properties}
There is an interesting connection from the rationality viewpoint between the complete intersection of three quadrics $\QQ \subset \PP^7$ and a divisor $X$ of bidegree $(2,2)$ in $\PP^2 \times \PP^3$. Both may be either rational or irrational, and the rational
ones are dense in moduli. The similarities do not end here, since both varieties are birational to a quadric surface bundle over $\PP^2$ ramified over an octic curve, and they both (birationally) specialize to a singular Fano fourfold. These phenomena are thoroughly studied in \cite{hpt18a, hpt18b}. Unsurprisingly, a FK3 4-fold birational to a (singular) $X$ appears in our list of FK3 \cite{bfmt}, namely, \textbf{R-62}. Its rationality properties are unknown, as in the case of \textbf{C-8, R-63} (for which we know that their derived category contain a copy of the category of a K3 surface, together with a Brauer twist in the first case). On the other hand, \textbf{K3-27} is always rational, being birational to the complete intersection of two quadrics.
This shows how the behaviour of \emph{singular} $\QQ$ might differ a lot from the smooth case: it would be then interesting to determine the rationality properties of the above 4-folds, especially in light of Kuznetsov's conjecture for cubic and Gushel-Mukai fourfolds, see, e.g. \cite{kp18}.

\vspace{-0.1cm}
\subsection{An example}
We will now focus on a specific example appearing above, with the aim of giving a geometric interpretation to our \emph{singular K3 structure} in a specific case. In order to do this, let us focus on one of the case appearing above, namely \textbf{K3-27} in the notation of \cite{bfmt}. Again, this can be realized as $X=(\PP^6 \times \PP^7, \mQ_{\PP^6}(0,1) \oplus \of(2,0)^{\oplus 2} \oplus \of(1,1))$. In other words, we are considering the complete intersection $\QQ \subset \PP^7$ of three quadrics $Q_1, Q_2, Q_3$, such that $\mathrm{sing (Q_1)}=\mathrm{sing (Q_2)}=p$ and $Q_3$ is smooth, and we are blowing up $\Bl_p \QQ$.
The cohomology of the complete intersection of three quadrics in an odd-dimensional projective space  $\PP^{2n+1}$ is a classical topic: the original argument by O'Grady can be adapted to the singular case as well, see \cite{k1,k2}. In general, consider the net $\Lambda \cong \PP^2$ of quadrics induced by $\lambda_1 Q_1+ \lambda_2 Q_2+ \lambda_3 Q_3$. We can consider a divisor $D \subset \PP^2$ of degree $2n+2$ parametrizing the singular quadrics in the net. We can consider the double cover $S \to \PP^2$ branched along $D$: this is a surface of general type whose $H^2$ is isomorphic to $H^{2n}(\QQ)$.

The same argument can be modified to our specific net of quadrics, with the caveat that of course $D$ here will not be necessarily smooth, or even irreducible. In particular, without loss of generality we can assume $p=v_0$: our quadrics will be written as:
\[
Q_1= \sum_{i,j=1}^7 q^1_{i,j} x_ix_j; \ \ Q_2= \sum_{i,j=1}^7 q^2_{i,j} x_ix_j; \ \ \ Q_3= \sum_{i,j=0}^7 q^3_{i,j} x_ix_j, \ q^3_{0,0}=0.
\]

The equation for $D$ is in general given by $\det(M)$, where (up to multiplying everything by 2) $M_{i,j}= \lambda_1 q^1_{i,j}+\lambda_2 q^2_{i,j}+\lambda_3 q^3_{i,j}$. We have the following immediate result.
\begin{proposition} The branch divisor $D \subset \Lambda$ associated to the above $\QQ$ is the reducible union of a double line $L$ and a smooth sextic $C$.
\end{proposition}
\begin{proof} It is a simple application of Laplace's rule for the computation of determinants. In fact, by our special choice of quadrics, we will have $M_{0,0}=0$ and $M_{0, i}= M_{i,0}= q^3_{0,i} \lambda_3$. Expanding using the first row, we have $\det(M)= \sum_i (-1)^i q^3_{0,i} \lambda_3 \det(M^i)$, with $M^i$ denoting the matrix of cofactors. On the other hand, each of $\det(M^i)$ can be further expanded using the first column, which once again has only terms in $\lambda_3$. In particular, since $q^3_{0,0}=0$, each monomial in $\det(M^i)$ will contain a $\lambda_3$. Hence, we can write $\det(M)=\lambda_3^2 f_6(\lambda_1, \lambda_2, \lambda_3)$, with the degree 6 polynomial generically irreducible (and the corresponding curve smooth).
\end{proof}

The above Proposition gives a geometric explanation of where the K3 structure in $\QQ$ comes from, namely the sextic curve $C \subset \Lambda$, as of course the double cover of $\PP^2$ ramified in $C$ is a degree 2 K3 surface. In general, it is not hard to imagine that all other degenerations of $\QQ$ producing a K3 structure on the resolution will come from a similar situation. Classifying such \emph{FK3-degenerations} becomes then an interesting problem \emph{per se}, especially if one would be able to associate to a FK3-degeneration a projective family of hyperk\"ahler manifolds. What is even more interesting of course would be to produce other examples of FK3-degenerations where no actual K3 surfaces are involved, and possibly construct family of hyperk\"ahler from them. This, however, goes well beyond the scope of this short survey.

\vspace{-0.5cm}
\thispagestyle{empty}


\frenchspacing


\newcommand{\etalchar}[1]{$^{#1}$}

\end{document}